\documentclass[a4paper,10pt]{article}
\title{Random real branched coverings of the projective line.}
\author{Michele Ancona\,\thanks{Tel Aviv University, School of Mathematical Sciences; e-mail: \url{ancona@math.univ-lyon1.fr}. The author is supported by the Israeli Science Foundation through the ISF Grants 382/15 and 501/18.}} 
\date{}
\usepackage[a4paper,bindingoffset=0.2in,%
            left=1in,right=1in,top=1in,bottom=1in,%
            footskip=.25in]{geometry}
\usepackage[latin1]{inputenc}
\usepackage[T1]{fontenc}
\usepackage[english]{babel}
\usepackage{amsmath,amssymb,amsthm}
\usepackage{amsfonts}%
\usepackage{amssymb}%
\usepackage{indentfirst}
\usepackage{fancyhdr}
\usepackage{physics}
\usepackage{young}
\usepackage{graphicx}

\usepackage[vcentermath]{youngtab}

\RequirePackage[colorlinks,linkcolor=blue,citecolor=blue,urlcolor=blue]{hyperref}
\usepackage{enumerate}
\usepackage{color}
\usepackage{pst-fill,pst-grad,pst-plot,pst-eucl,pstricks-add,pst-node}
\usepackage{mathdots}
\usepackage{dsfont}

\theoremstyle{plain}
\newtheorem{thm}{Theorem}[section]
\newtheorem{lemma}[thm]{Lemma}
\newtheorem{prop}[thm]{Proposition}
\newtheorem{cor}[thm]{Corollary}
\newtheorem{oss}[thm]{Remark}
\newtheorem{example}[thm]{Example}

\theoremstyle{definition}
\newtheorem{defn}[thm]{Definition}

\newcommand{\R}{\mathbb{R}}
\newcommand{\C}{\mathbb{C}}

\newcommand{\Crit}{\textrm{Crit}}
\newcommand{\Vol}{\textrm{Vol}}

\newcommand{\Fix}{\textrm{Fix}}
\newcommand{\conj}{\textit{conj}}
\newcommand{\Pic}{\textrm{Pic}}
\newcommand{\dH}{\textrm{dH}}

\allowdisplaybreaks

\begin{document}
\maketitle
\begin{abstract} In this paper, we construct a natural probability measure on the space of real branched coverings from a real projective algebraic curve $(X,c_X)$ to the projective line $(\C\mathbb{P}^1,\conj)$. We prove that the space of degree $d$ real branched coverings having "many" real branched points (for example more than $\sqrt{d}^{1+\alpha}$, for any $\alpha>0$) has exponentially small measure. In particular, maximal real branched coverings, that is  real branched coverings such that all the branched points are real, are exponentially rare.
\end{abstract}
\section*{Introduction}
Let $(X,c_X)$ be a real algebraic curve, that is a smooth complex complex  curve equipped with  an anti-holomorphic involution $c_X$, called the real structure. We denote by $\R X$ the real locus of $X$, that is the set $\textrm{Fix}(c_X)$ of fixed points of $c_X$. For example the projective line $(\C\mathbb{P}^1,\conj)$ is a real algebraic curve whose real locus equals $\R \mathbb{P}^1$.

The central objects of this paper are real branched coverings from $X$ to $\C\mathbb{P}^1$, that is, the branched coverings   $u:X\rightarrow \C\mathbb{P}^1$ such that $u\circ c_{X}=\conj\circ u$. Let us denote by  $\mathcal{M}^{\R}_d(X)$ the set of degree $d$ real branched coverings from $X$ to $\C\mathbb{P}^1$. 
The first purpose of the paper is to show that  $\mathcal{M}^{\R}_d(X)$ has a natural probability measure $\mu_d$ induced by a compatible volume form $\omega$ of $X$ (that is $c_{X}^*\omega=-\omega$), which we fix once for all. Later in the introduction we will sketch the construction of the measure $\mu_d$, which we will give in details in Section \ref{sectproba}. 
By Riemann-Hurwitz formula, the number of critical points, counted with multiplicity, of a degree $d$ branched covering $u:X\rightarrow \C\mathbb{P}^1$ equals $2d+2g-2$, where $g$ is the genus of $X$. The probability measure $\mu_d$ allows us to ask the following question. 
\begin{center}
What is the probability that all the critical points of a real branched covering $u\in\mathcal{M}^{\R}_d(X)$ are real?
\end{center}
In \cite{anc}, it is proved that the expected number of real critical points is equivalent to $c\sqrt{d}$ as the degree $d$ of the random branched covering goes to infinity. The constant $c$ is  explicit, given by $c=\sqrt{\frac{\pi}{2}}\Vol(\R X)$, where $\Vol(\R X)$ is the length of the real locus of $X$ with respect to the Riemannian metric induced by $\omega$.
The main theorem of the paper is the following exponential rarefaction  result for real branched coverings having "many" real critical points. 
\begin{thm}\label{rarefaction} Let $X$ be a real algebraic curve. Let $\ell(d)$ be a sequence of positive real numbers such that $\ell(d)\geq B \log d$, for some $B>0$. Then there exist positive constants $c_1$ and $c_2$ such that the following holds 
$$\mu_d\big\{u\in \mathcal{M}_d^{\R}(X), \#(\Crit(u)\cap \R X)\geq \ell(d)\sqrt{d}\big\}\leq c_1e^{-c_2\ell(d)^2}.$$
\end{thm} 
For example, for any fixed $\alpha>0$, we can consider the sequence $\ell(d)=\sqrt{d}^{\alpha}$. Theorem \ref{rarefaction} says that the space of real branched coverings having more than $\sqrt{d}^{1+\alpha}$ real critical points has exponential small measure. In particular, maximal real branched converings (i.e. branched coverings such that all the critical points are real) are exponentially rare.

\paragraph{The probability measure on $\mathcal{M}_d^{\R}(X)$.} The construction of the probability measure on $\mathcal{M}_d^{\R}(X)$ uses the fact that there is a  natural map from  $\mathcal{M}_d^{\R}(X)$ to the space of degree $d$ real holomorphic line bundle $\Pic^d_{\R}(X)$, see Proposition \ref{natmor}. This  map sends a degree $d$ morphism $u$ to the degree $d$ line bundle $u^*\mathcal{O}(1)$.  The fiber  of this map over $L \in\Pic^d_{\R}(X)$  is the  open dense subset of $\mathbb{P}(\R H^0(X,L)^2)$ given by (the class of) pairs of global sections without common zeros.  In order to construct a probability measure on $\mathcal{M}^{\R}_d(X)$, we produce a family of probability measures $\{\mu_{L}\}_{L\in\Pic_{\R}^d(X)}$ on each space  $\mathbb{P}(\R H^0(X,L)^2)$.  The probability measure $\mu_{L}$ on $\mathbb{P}(\R H^0(X,L)^2)$ is the measure induced by the Fubini-Study metric associated with a real Hermitian product on $\R H^0(X,L)^2$. This Hermitian product is the natural $\mathcal{L}^2$-product induced by $\omega$, see Section \ref{background}. This family of measures, together with the Haar probability measure on the base $\Pic^d_{\R}(X)$, gives rise to the probability measure $\mu_d$ on $\mathcal{M}_d^{\R}(X)$.

\paragraph{An example: the projective line.}  Let us consider the case $X=\mathbb{C}\mathbb{P}^1$ equipped with the conjugaison $\conj([x_0:x_1])=[\bar{x}_0:\bar{x}_1]$. Given two degree $d$ real polynomials $P,Q\in \R_d^{hom}[X_0,X_1]$ without common zeros, we produce a degree $d$ real branched covering $u_{PQ}:\mathbb{C}\mathbb{P}^1\rightarrow\mathbb{C}\mathbb{P}^1$ by sending $[x_0:x_1]$ to $[P(x_0,x_1):Q(x_0,x_1)]$. We also remark that the  pair $(\lambda P,\lambda Q)$ defines the same branched covering. Conversely, one can prove that any degree $d$ real branched covering $u:\C\mathbb{P}^1\rightarrow \C\mathbb{P}^1$ is of the form $u=u_{PQ}$ for some (class of) pair of polynomials $(P,Q)$ without common zeros.  This means that $\mathcal{M}^{\R}_d(\C\mathbb{P}^1)=\mathbb{P}(\R_d^{hom}[X_0,X_1]^2\setminus \Lambda_d)$, where $\Lambda_d$ is the set of polynomials with at least one common zero. Consider the affine chart $\{x_1\neq 0\}$, the corresponding coordinate $x=\frac{x_0}{x_1}$  and the polynomials $p(X)=P(X_0,1)$ and $q(X)=Q(X_0,1)$. Then, one can see that a point $x\in \{x_1\neq 0\}$ is a critical point of $u_{PQ}$ if and only $p'(x)q(x)-q'(x)p(x)=0$ (see Proposition \ref{wronskiancritical}).

 In the the previous paragraph, we constructed a probability measure on this space by fixing a compatible volume form on source space, in this case $\C\mathbb{P}^1$. Indeed, a compatible volume form induces a $\mathcal{L}^2$-scalar product on $\R_d^{hom}[X_0,X_1]$ which will induce a Fubini-Study volume form on $\mathbb{P}(\R_d^{hom}[X_0,X_1]^2)$ and then a probability on  $\mathcal{M}^{\R}_d(\C\mathbb{P}^1)$. If we equip $\mathbb{C}\mathbb{P}^1$ with the Fubini-Study form, then the induced scalar product on $\R_d^{hom}[X_0,X_1]$ is the one which makes $\{\sqrt{\binom{d}{k}}X_0^kX_1^{d-k}\}_{0\leq k\leq d}$ an orthonormal basis. This scalar product was considered by Kostlan in \cite{ko} (see also \cite{ss}). It is the only scalar product invariant under the action of the orthogonal group $O(2)$ (which acts on the variables $X_0$ and $X_1$) and such that the standard monomials are orthogonal to each other.

\paragraph{About the proof.} There are two main steps in the proof of our main theorem.  First, we reduce our problem into the problem of the computation of the \emph{Gaussian} measure of a cone $\mathcal{C}_{\ell(d)}$ which lies inside the space of pairs of global sections of a real holomorphic line bundle over $X$. This cone is  defined by using the Wronskian of a pair of  global sections, which plays a key role. Then,  we use peak sections theory  to estimate some Markov moments related to this Wronskian. These moments, together with Poincar\'e-Lelong formula, allow us to estimate the measure of the cone $\mathcal{C}_{\ell(d)}$.

Let sketch the proof  in more details.
We fix  a degree $1$ real holomorphic line bundle $F$ over $X$, so that, for any $L\in\Pic^d_{\R}(X)$ there exists an unique $E\in\Pic^0_{\R}(X)$ such that $L=F^d\otimes E$. Recall that any class of pairs of real global sections without common zeros $[\alpha:\beta]\in \mathbb{P}(\R H^0(X,F^d\otimes E)^2)$  defines a real branched covering $u_{\alpha\beta}$ by sending a point $x\in X$ to $[\alpha(x):\beta(x)]\in\C\mathbb{P}^1$.  Theorem \ref{rarefaction} will follow from the estimate
\begin{equation}\label{about1}
\mu_{F^d\otimes E}\big\{[\alpha:\beta]\in \mathbb{P}(\R H^0(X,F^d\otimes E)^2), \#(\Crit(u_{\alpha\beta})\cap \R X)\geq \ell(d)\sqrt{d}\big\}\leq c_1e^{-c_2\ell(d)^2}
\end{equation}
where $\mu_{F^d\otimes E}$ is the probability measure induced by the Fubini-Study metric on $\mathbb{P}(\R H^0(X,F^d\otimes E)^2)$. Indeed, if we integrate the inequality \eqref{about1} along $\Pic^0_{\R}(X)$ we exactly obtain  Theorem \ref{rarefaction}.
To prove the estimate \eqref{about1}, we will use the following two facts. First, a point $x$ is a critical point of $u_{\alpha\beta}$ if and only if it is a zero of the Wronskian $W_{\alpha\beta}\doteqdot\alpha\otimes\nabla\beta-\beta\otimes\nabla\alpha$. Second, the pushforward (with respect to the projectivization) of the Gaussian measure on $\R H^0(X,F^d\otimes E)^2$ is exactly the probability measure  $\mu_{F^d\otimes E}$. These two facts imply that the estimate \eqref{about1} is equivalent to the fact that the Gaussian measure  of the cone
\begin{equation}\label{about2}
\mathcal{C}_{\ell(d)}\doteqdot\big\{(\alpha,\beta)\in \R H^0(X,F^d\otimes E)^2, \#(\textrm{real zeros of}\hspace{1mm}W_{\alpha\beta})\geq \ell(d)\sqrt{d}\big\}
\end{equation}
is bounded from above by $c_1e^{-c_2\ell(d)^2}$. \\
In order to estimate the Gaussian measure of $\mathcal{C}_{\ell(d)}$, inspired by \cite{gwexp},   we  bound from above the moments of the random variable 
$(\alpha,\beta)\in \R H^0(X,F^d\otimes E)^2\mapsto \log \norm{W_{\alpha\beta}(x)}$, where $x$ is a point in $X$ such that $\textrm{dist}(x,\R X)$ is bigger that $\frac{\log d}{\sqrt{d}}$, see Proposition \ref{moments}.  This condition on the distance is  natural, it is strictly related to peak section's theory  (see \cite{tian,gw1}) and it is the reason why we need the hypothesis on the growth of the sequence $\ell(d)$ in Theorem \ref{rarefaction}. The estimate of these moments uses two ingredients: the theory of peak sections and the comparison between the norms of two differents evaluation maps (and more generally jet maps).
Once   these moments are estimates,  Markov inequality and Poincar\'e-Lelong formula gives us the exponential rarefaction of the Gaussian measure of the cone (\ref{about2}).

\paragraph*{Organization of the paper.} The paper is organized as follows. In Section \ref{background} we introduce the main objects and notations of this paper. In Sections \ref{secspaceofbranched} and \ref{sectproba} we study the geometry of the manifold $\mathcal{M}^{\R}_d(X)$  and we construct the probability measure $\mu_d$ on it.\\
The purpose of  Section \ref{secmoment} is to prove Proposition \ref{moments}, that is, to estimate of the moments of the random variable $(\alpha,\beta)\in \R H^0(X,F^d\otimes E)^2\mapsto \log \norm{(\alpha\otimes\nabla\beta-\beta\otimes\nabla\alpha)(x)}$, for $F$ and  $E$ are respectively a degree $1$ and $0$ real holomorphic line bundles. In order to do this, in Section \ref{secgauss} we introduce Gaussian measures on  $\R H^0(X,F^d\otimes E)^2$  and in Section \ref{secpeak} we study jet maps   at points $x\in X$ which are far from the real locus.
Finally, in Section \ref{secproof}, we deduce Theorem \ref{rarefaction} from the estimates established in Section \ref{secmoment}.
\section[Random coverings]{Random real branched coverings}\label{ramframwork}

\subsection{Background}\label{background}
Let $(X,c_X)$ be a real algebraic curve, that is a complex, projective, smooth curve equipped with  an anti-holomorphic involution $c_X$, called the real structure. We  assume that the real locus $\R X\doteqdot\Fix(c_X)$  is non empty. For example $(\C\mathbb{P}^1,\conj)$, where $\conj([x_0:x_1])=[\bar{x}_0:\bar{x}_1]$, is a real algebraic curve whose real locus is $\R\mathbb{P}^1$.
 A real holomorphic line bundle $p:L\rightarrow X$  is a line bundle  equipped with an anti-holomorphic involution $c_{L}$ such that $p\circ c_X=c_{L}\circ p$ and $c_{L}$ is complex-antilinear in the fibers. 
We denote by $\mathbb{R} H^0(X;L)$ the real vector space of real holomorphic global sections of $L$, i.e. sections $s\in H^0(X;L)$ such that $s \circ c_X=c_{L}\circ s$. Let  $\Pic_{\R}^d(X)$ be the set of degree $d$ real line bundles. It is a principal space under the action of the compact topological abelian group $\Pic_{\R}^0(X)$ and so it inherits a normalized Haar measure that we denote by  $\dH$ (see, for example, \cite{realalgebraiccurves}). 
 Finally, recall  that a real Hermitian metric $h$ on $L$ is a Hermitian metric  on $L$ such that $c_{L}^*h=\bar{h}$.
\begin{prop}\label{metr} Let $(X,c_X)$ be a real algebraic curve and let $\omega$  be a compatible volume form of mass $1$, that is $c_{X}^*\omega=-\omega$ and $\int_X\omega=1$. Let $L\in \Pic^d_{\R}(X)$ be a degree $d$ real holomorphic line bundle over $X$, then there exists an unique  real Hermitian metric $h$ (up to multiplication by a positive real constant) such that $c_1(L,h)=d\cdot\omega$. 
\end{prop}
\begin{proof} For the existence and unicity of such metric, see \cite[Proposition 1.4]{anc3}. The fact that the metric $h$ is real follows from the  following argument. Let us consider the Hermitian metric $\overline{c^*_{L}h}$ on $L$. Claim: its curvature equals $-d\cdot c_X^*\omega$. Indeed, for any $x\in X$ we  consider a real meromorphic section $s$ of $L$ such that $x$ and $c_{X}(x)$ are neither zero nor pole of $s$ (such section exists by Riemann-Roch Theorem).  Then, the curvature of $(L,\overline{c^*_{L}h})$ around $x$ is $\partial\bar{\partial}\log (\overline{c^*_{L}h})_x(s(x),s(x))=\partial\bar{\partial}\log h_{c_X(x)}(c_{L}(s(x)),c_{L}(s(x)))=\partial\bar{\partial}\log h_{c_X(x)}(s(c_X(x)),s(c_X(x)))=\partial\bar{\partial} c^*_X\log h (s,s)=-c^*_X\partial\bar{\partial}\log h (s,s),$ where the last equality is due to the anti-holomorphicity of  $c_X$. Then, the claim follows from the fact that $\partial\bar{\partial}\log h (s,s)=d\cdot\omega$. \\
 Now, consider  the real Hermitian metric $(h\cdot\overline{c^*_{L}h})^{1/2}$. Its curvature equals $$\frac{1}{2}\big(\partial\bar{\partial}\log h(s,s)+\partial\bar{\partial}\log (\overline{c^*_{L}h})(s,s)\big)=\frac{1}{2}(d\cdot\omega-d\cdot c_X^*\omega)=d\cdot\omega,$$ where the last equality follows from the fact that  $\omega$ is compatible with the real structure. By unicity of the metric with  curvature $d\cdot\omega$, this implies that $(h\cdot\overline{c^*_{L}h})^{1/2}$ is a multiple of $h$. We actually have the equality $(h\cdot\overline{c^*_{L}h})^{1/2}=h$, because for a real point $x\in\R X$ and a real vector $v\in \R L_x$ we get $(h_x(v,v)\cdot(\overline{c^*_{L}h})_x(v,v))^{1/2}=(h_x(v,v)\cdot h_x(v,v))^{1/2}=h_x(v,v)$. 
\end{proof}
\begin{defn}\label{scalarprod} Let  $\omega$  be a compatible volume form of mass $1$, let  $L\in \Pic^d_{\R}(X)$ be a degree $d$ line bundle over $X$ and $h$ be the real Hermitian metric given by the previous proposition. We define the $\mathcal{L}^2$-scalar product on $\R H^0(X;L)$ by
$$\langle\alpha,\beta\rangle_{\mathcal{L}^2}=\int_{x\in X}h_x(\alpha(x),\beta(x))\omega$$
for any pair of real holomorphic sections $\alpha,\beta\in\R H^0(X;L).$
\end{defn}

\subsection{The space of real branched coverings}\label{secspaceofbranched} 
In this section we introduce and study the space of real branched coverings from a real algebraic curve $(X,c_X)$ to $(\C\mathbb{P}^1,\conj)$.
\begin{defn} We denote by $\mathcal{M}^{\R}_d(X)$ the space of all degree $d$ real branched coverings $u:X\rightarrow \C\mathbb{P}^1$, that are the branched coverings    such that $u\circ c_{X}=\conj\circ u$.
\end{defn}
A natural way to define a degree $d$  real branched covering  is the following one. Consider a degree $d$ real holomorphic line bundle $L\in \Pic^d_{\R}(X)$ and two real holomorphic sections $\alpha,\beta\in \R H^0(X,L)$ without common zeros. Then, we can define the degree $d$ real branched covering $u_{\alpha\beta}$ defined by $$u_{\alpha\beta}:x\in X\mapsto [\alpha(x):\beta(x)]\in\C\mathbb{P}^1.$$
\begin{prop}\label{realproj} Two pairs  $(\alpha,\beta), (\alpha',\beta')$ of real holomorphic sections of $L$ define the same real branched covering if and only if  $(\alpha',\beta')=(\lambda\alpha,\lambda\beta)$ for some $\lambda\in\R^*$.
\end{prop}
\begin{proof}
The proof follows the lines of \cite[Proposition 1.1]{anc3}.
\end{proof}
\begin{prop}\label{natmor} There exists a natural map from $\mathcal{M}^{\R}_d(X)$ to the space $\Pic_{\R}^d({X})$ of degree $d$ real line bundles   over $X$. This natural map  is given by $u\in\mathcal{M}^{\R}_d(X)\mapsto u^*\mathcal{O}(1)\in\Pic_{\R}^d({X})$. The fiber   over $L\in\Pic_{\R}^d({X})$ is the  open subset of $\mathbb{P}(\R H^0(X;L)^2)$  given by (the class of) pair of sections $(\alpha,\beta)$ without common zeros.
\end{prop}
\begin{proof}
Given a degree $d$ real branched covering $u:X\rightarrow\C\mathbb{P}^1$, we get a degree $d$ real line bundle   $u^*\mathcal{O}(1)$ over $X$  and the class of two real holomorphic global holomorphic sections without common zeros $[u^*x_0:u^*x_1]\in \mathbb{P}(\R H^0(X;u^*\mathcal{O}(1))^2)$. On the other hand, given a degree $d$ real line bundle $L\rightarrow X$ and two real holomorphic global sections without common zeros $(\alpha,\beta)\in \R H^0(X;L)^2$, then the map $u_{\alpha\beta}:X\rightarrow\C\mathbb{P}^1$ defined by $x\mapsto [\alpha(x):\beta(x)]$ is a degree $d$ real branched covering. Moreover, by Proposition \ref{realproj}, two pairs  $(\alpha,\beta)$ and  $(\alpha',\beta')$ of real holomorphic sections of $L$ define the same real branched covering if and only if  $(\alpha',\beta')=(\lambda\alpha,\lambda\beta)$ for some $\lambda\in\R^*$, hence the result.
\end{proof}
\subsection{Probability on $\mathcal{M}^{\R}_d(X)$}\label{sectproba}
Let $X$ be a real algebraic curve equipped with a compatible volume form $\omega$ of total mass $1$.
In this section, we construct  a natural probability measure on the space $\mathcal{M}^{\R}_d(X)$ of degree $d$ real branched coverings from $X$ to $\mathbb{C}\mathbb{P}^1$.\\
Let $L\in\Pic^d_{\R}(X)$ be a degree $d$ real line bundle equipped with the real Hermitian metric $h$ given by Proposition \ref{metr}. We recall that in Definition \ref{scalarprod} we defined the  $\mathcal{L}^2$-scalar product  on the space $\R H^0(X;L)$ induced by the  Hermitian metric $h$. This $\mathcal{L}^2$-scalar product  induces a scalar product on the Cartesian product  $\R H^0(X;L)^2$ and then  a Fubini-Study metric on  $\mathbb{P}(\R H^0(X;L)^2)$.
We recall that the Fubini-Study metric is constructed as follows. First, we restrict the scalar product to the unit sphere of $\R H^0(X;L)^2$. The obtained metric is invariant under the action of $\mathbb{Z}/2\mathbb{Z}$ and the Fubini-Study metric is then the quotient metric on $\mathbb{P}(\R H^0(X;L)^2)$. 
\begin{defn}\label{measurefiber} Let $L$ be a real holomorphic line bundle over $X$. We denote by  $\mu_{L}$ the probability measure on  $\mathbb{P}(\R H^0(X;L)^2)$ induced by the normalized Fubini-Study volume form. Here, the Fubini-Study metric on $\mathbb{P}(\R H^0(X;L)^2)$ is the one induced by the Hermitian metric on $L$ given by Proposition \ref{metr}.
\end{defn}
\begin{prop}\label{fs} The probability measure $\mu_{L}$ over  $\mathbb{P}(\R H^0(X;L)^2)$ does not depend on the choice of the multiplicative constant in front  of the metric $h$ given by Proposition \ref{metr}.
\end{prop}
\begin{proof}
The proof follows the line of \cite[Proposition 1.7]{anc3}
\end{proof}
\begin{oss}
For a real holomorphic line bundle $L$, we denote by $\Lambda_L$  the space of pair of sections $(s_0,s_1)\in \R H^0(X;L)^2$ with at least a common zeros. By \cite[Proposition 2.11]{anc}, the set  $\Lambda_L$ has zero measure (it is an hypersurface), at least if the degree of $L$ is large enough. This implies  that $\mu_L$ induces a probability measure on $\mathbb{P}( \R H^0(X;L)^2\setminus \Lambda_L)$, still denoted by $\mu_L$.  
\end{oss}
\begin{defn}\label{rhfs} We define the  probability measure $\mu_d$  on $\mathcal{M}^{\R}_d(X)$ by the following equality:
$$\int_{\mathcal{M}^{\R}_d(X)}f\textrm{d}\mu_d=\int_{L\in\Pic_{\R}^d(X)}\left(\int_{\mathcal{M}^{\R}_d(X,L)}f\textrm{d}\mu_{L}\right)\dH(L)$$
for any $f\in \mathcal{M}^{\R}_d(X)$ measurable function. Here: 
\begin{itemize}
\item $\mathcal{M}^{\R}_d(X,L)$ is the fiber of the natural morphism $\mathcal{M}^{\R}_d(X)\rightarrow \Pic^d_{\R}(X)$ defined in Proposition \ref{natmor}.
\item  $\mu_{L}$ denotes (by a slight abuse of notation)  the restriction to $\mathcal{M}^{\R}_d(X,L)$ of the probability measure on $\mathbb{P}(\R H^0(X,L)^2)$ defined in Definition \ref{measurefiber}.  
\item $\dH$ denotes the normalized Haar measure on $\Pic^d_{\R}(X)$.
\end{itemize}
\end{defn}
\begin{oss}  The probability measure $\mu_d$ of Definition \ref{rhfs} is the real analogue of the one constructed in the complex setting in \cite{anc3} for the study of random branched coverings from a fixed Riemann surface to $\C \mathbb{P}^1$. Also, in the complex setting, a similar construction has been considered by Zelditch in \cite{zeldlarge}  in order to study large deviations of empirical measures of zeros on a Riemann surface.
\end{oss}

\begin{example} Let us consider the case  $(X,c_X)=(\C\mathbb{P}^1,\conj)$, where $\C\mathbb{P}^1$ is equipped with the Fubini-Study form $\omega_{FS}$. For the projective line $\C\mathbb{P}^1$, the unique degree $d$ real line bundle is  the line bundle $\mathcal{O}(d)$, which is naturally equipped with a real Hermitian metric $h_{d}$ whose curvature equals $d\cdot\omega$. The space of real holomorphic global sections $\R H^0(\C\mathbb{P}^1;\mathcal{O}(d))$ is isomorphic to the space of degree $d$ homogeneous polynomials $\R^{\textrm{hom}}_d[X_0,X_1]$ and the $\mathcal{L}^2$-scalar product coincides with the Kostlan scalar product (i.e. the scalar product which makes $\{\sqrt{\binom{d}{k}}X_0^kX_1^{d-k}\}_{0\leq k \leq d}$ an orthonormal basis, see \cite{ko,ss}). Then, a random  real branched covering $u:\C\mathbb{P}^1\rightarrow\C \mathbb{P}^{1}$ is  given by the class of a pairs of independent Kostlan polynomials. 
\end{example}

\section{Gaussian measures and estimates of higher moments}\label{secmoment}
In this section, we introduce some Gaussian measures on the spaces $\R H^0(X;L)^2$ and $H^0(X;L)^2$, as in \cite{gw1,gwexp,sz,anc}. We follow the notations of Section \ref{ramframwork}. In particular, $(X,c_X)$ is a real algebraic curve whose real locus $\R X$ is not empty.
\subsection{Gaussian measures}\label{secgauss}  In this section, given any degree $d$ real line  $L\in\Pic_{\R}^d(X)$, we equip the cartesian product  $\R H^0(X;L)^2$ of the space of real holomorphic section with  a Gaussian measure  $\gamma_{L}$. In order to do this, 
we fix a  compatible volume  form $\omega$ of total volume $1$  (i.e. $c_X^*\omega=-\omega$ and $\int_X\omega=1$). Given $L\in\Pic_{\R}^d(X)$,  we equip $L$ by the real Hermitian metric $h$ with curvature $d\cdot\omega$ (the metric $h$ is unique up to a multiplicative constant, see Proposition \ref{metr}).\\
In Definition \ref{scalarprod}, we defined  a $\mathcal{L}^2$-Hermitian product on the space $\R H^0(X;L)$ of real holomorphic global holomorphic sections of $L$ denoted by $\langle\cdot,\cdot \rangle_{\mathcal{L}^2}$ and  defined by 
$$\langle\alpha,\beta\rangle_{\mathcal{L}^2}=\int_{x\in X}h_x(\alpha(x),\beta(x))\omega$$
for all $\alpha,\beta$ in $\R H^0(X;L)$. 
\begin{defn} The $\mathcal{L}^2$-scalar product on $\R H^0(X;L)^2$ induces a  Gaussian measure $\gamma_{L}$ on  $\R H^0(X;L)^2$  defined  by
$$\gamma_{L}(A)=\frac{1}{\pi^{N_d}}\int_{(\alpha,\beta)\in A}e^{-\norm{\alpha}_{\mathcal{L}^2}^2-\norm{\beta}_{\mathcal{L}^2}^2}\textrm{d}\alpha \textrm{d}\beta$$
for any open subset  $A\subset \R H^0(X;L)^2$. Here  $\textrm{d}\alpha\textrm{d}\beta$ is the Lebesgue measure on $(\R H^0(X;L)^2;\langle\cdot,\cdot\rangle_{\mathcal{L}^2})$ and $N_d$ denotes the dimension of $\R H^0(X;L)$, which equals the complex dimension of $H^0(X;L)$. 
\end{defn}
\begin{oss}
If $d>2g-2$, where $g$ is the genus of $X$, then $H^1(X;L)=0$ and then, by Riemann-Roch theorem, we have  $N_d=d+1-g$.
\end{oss}
\begin{prop}\cite[Proposition 1.12]{anc3}\label{vs} Let $f$ be a function on an Euclidian space $(V,\langle\cdot,\cdot\rangle)$ which is constant over the lines, i.e. $f(v)=f(\lambda v)$ for all $v\in V$ and all $\lambda\in\R^*$. Denote  by $d\gamma$  the Gaussian measure on $V$ induced by $\langle\cdot,\cdot\rangle$ and by $d\mu$  the normalized Fubini-Study measure on the projectivized $\mathbb{P}(V)$.
Then, for all cones $A\subset V$, we have 
$$\int_{A}f\textrm{d}\gamma=\int_{\mathbb{P}(A)}[f]{d}\mu$$
where $\mathbb{P}(A)$ is the projectivized of $A$ and $[f]$ is the function on $\mathbb{P}(V)$ induced by $f$. 
\end{prop}

We will be also interested in the \emph{complex} Gaussian measure on the space $H^0(X,L)^{2}$. Indeed, the Hermitian metric $h$ on $L$ defines a $\mathcal{L}^2$-Hermitian product on $H^0(X,L)$ by the formula $$\langle\alpha,\beta\rangle_{\mathcal{L}^2}=\int_{x\in\Sigma}h_x(\alpha(x),\beta(x))\omega$$
for all $\alpha,\beta$ in $H^0(\Sigma;L)$.
\begin{defn}\label{complexgaussian}  The complex Gaussian measure $\gamma^{\C}_{L}$ on  $H^0(\Sigma;L)^2$ is defined  by
$$\gamma^{\C}_{L}(A)=\frac{1}{\pi^{2N_d}}\int_{(\alpha,\beta)\in A}e^{-\norm{\alpha}_{\mathcal{L}^2}^2-\norm{\beta}_{\mathcal{L}^2}^2}\textrm{d}\alpha \textrm{d}\beta$$
for any open subset  $A\subset H^0(\Sigma;L)^2$. Here  $\textrm{d}\alpha\textrm{d}\beta$ is the Lebesgue measure on $(H^0(\Sigma;L)^2;\langle\cdot,\cdot\rangle_{\mathcal{L}^2})$ and $N_d$ denotes the complex dimension of $H^0(\Sigma;L)$.
\end{defn}
\subsection{Jet maps and peak sections}\label{secpeak}
Let $F$ and $E$ be respectively  degree $1$ and $0$ real holomorphic line bundles over $X$.
We equip $F$ and $E$ by the real  Hermitian metrics given by Proposition \ref{metr} which we denote by $h_{F}$ and $h_{E}$. In particular the real Hermitian metric $h_d\doteqdot h_{F}^d\otimes h_{E}$ on $F^d\otimes E$ is such that its curvature equals $d\cdot\omega$. Finally, recall that the space $H^0(X,F^d\otimes E)$ is endowed with the $\mathcal{L}^2$-Hermitian product $$\langle\alpha,\beta\rangle_{\mathcal{L}^2}=\int_{x\in X}h_d(\alpha(x),\beta(x))\omega$$ defined by
for any $\alpha,\beta$ in $H^0(X;F^d\otimes E).$ 

\begin{defn}\label{defn eval} For any  $x\in X$, let $H_x$ be the kernel of the  map  ${s\in  H^0(X,F^d\otimes E)\mapsto s(x)\in (F^d\otimes E)_x}$. 
Similarly, we denote by $H_{2x}$ the kernel of the map   $s\in H_{x}\mapsto\nabla s(x)\in (F^d\otimes E)_x\otimes T^*_{X,x}.$
We define the following jet maps:
$$ev_x:s\in H^0(X,F^d\otimes E)/H_x\mapsto s(x)\in(F^d\otimes E)_x,$$
$$ev_{2x}:s\in H_x/H_{2x}\mapsto \nabla s(x)\in (F^d\otimes E)_x\otimes T^*_{X,x}.$$
\end{defn}
The previous definition has the following real analogue:
\begin{defn}\label{defn real eval}
For any point $x\in X$, we define the real  vector spaces $\R H^0_x=H^0_x\cap \R H^0(X, F^d\otimes E)$ and $\R H^0_{2x}=H^0_{2x}\cap \R H^0(X, F^d\otimes E)$
and the real jet maps by
 $$ev^{\R}_{x}:s\in \R H^0(X,F^d\otimes E)/\R H^0_x\mapsto s(x)\in(F^d\otimes E)_x,$$
 $$ev^{\R}_{2x}:s\in \R H_x/\R H_{2x}\mapsto \nabla s(x)\in (F^d\otimes E)_x\otimes T^*_{X,x}.$$
\end{defn}
By the fact that $F$ is ample (recall that $\deg{F}=1$), we get that for $d$ large enough the maps  $ev^{\R}_x$, $ev_x$, $ev^{\R}_{2x}$ and $ev_{2x}$ are invertible. The following proposition estimates the norms of this maps and of their inverses.
\begin{prop}\cite[Propositions 4 and 6]{gw1}\label{confronto evalu} For any $B>0$, then there exists an integer  $d_B$ and a positive constant $c_B$ such that,  for any $d\geq d_B$ and any point $x\in X$ with $\textrm{dist}(x,\R X)\geq B\frac{\log d}{\sqrt{d}}$, the maps $d^{-\frac{1}{2}}ev^{\R}_x$, $d^{-\frac{1}{2}}ev_x$, $d^{-1}ev^{\R}_{2x}$ and $d^{-1}ev_{2x}$ as well as their inverse have norms and determinants bounded from above by $c_B$.
\end{prop}
\begin{oss} In \cite[Propositions 4 and 6]{gw1}, the constant $B$ equals $1$, and the line bundle $E$ is trivial.  The same proof  actually holds for any fixed $B>0$ and any $E\in\Pic^0_{\R}(X)$. Indeed,  the proof is based on  the theory peak sections and Bergman kernels and this theory holds   in this more general setting (see for example \cite{daima} or \cite[Theorem 4.2.1]{ma2}).
\end{oss}
Using the $\mathcal{L}^2$-Hermitian product on $H^0(X,F^d\otimes E)$, we can identify $H^0(X,F^d\otimes E)/H_x$ with the orthogonal complement of $H_x$ in $H^0(X,F^d\otimes E)$. Similarly, we identify the quotient $H_x/H_{2x}$ with the orthogonal complement of $H_{2x}$ in $H_x$. We then have an orthogonal decomposition $$H^0(X,F^d\otimes E)=  H^0(X,F^d\otimes E)/H_x\oplus H_x/H_{2x} \oplus H_{2x}.$$ Similarly, using the $\mathcal{L}^2$-scalar product on  $\R  H^0(X,F^d\otimes E)$, we have the orthogonal decomposition $$\R H^0(X,F^d\otimes E)= \R H^0(X,F^d\otimes E)/\R H_x\oplus \R H_x/\R H_{2x} \oplus \R H_{2x}.$$ The map $ev_x\times ev_{2x}$  (resp. $ev^{\R}_x\times ev^{\R}_{2x}$) gives an isomorphism between $H^0(X,F^d\otimes E)/H_x\oplus H_x/H_{2x}$  (resp. $\R H^0(X,F^d\otimes E)/\R H_x\oplus \R H_x/\R H_{2x}$) and the fiber $(F^d\otimes E)_x\oplus (F^d\otimes E)_x\otimes T^*_{X,x}$. \\
Moreover, remark that we have natural identifications $H^0(X,F^d\otimes E)/H_x\oplus H_x/H_{2x}=H_{2x}^{\perp}$ and $\R H^0(X,F^d\otimes E)/\R H_x\oplus \R H_x/\R H_{2x}=\R H_{2x}^{\perp}$.
A direct consequence of Proposition \ref{confronto evalu} is the following
\begin{cor}\label{bounded determinant}
For any $B>0$,  there exist an integer  $d_B$ and a positive constant $c_B$ such that,  for any $d\geq d_B$ and any $x$ with $\textrm{dist}(x,\R X)\geq B\frac{\log d}{\sqrt{d}}$, the map $(ev^{\R}_x\times ev^{\R}_{2x})^{-1}\circ(ev_x\times ev_{2x}):H_{2x}^{\perp}\rightarrow \R H_{2x}^{\perp}$   has determinant bounded from above by $c_B$ and from below by $1/c_B$.
\end{cor}
\begin{defn}\label{defin peak} We denote by $s_0$ and $s_1$ the global holomorphic sections of $L^d\otimes E$ of unit $\mathcal{L}^2$-norm which generates respectively the orthogonal of $H_x$ in $H^0(X,F^d\otimes E)$ and the orthogonal of $H_{2x}$ in $H_x$.  We call these sections the \emph{peak sections} at $x$.
\end{defn}
 The pointwise estimate of the norms (with respect to the Hermitian metric $h_d$ of curvature $d\cdot\omega$) of the peak sections are well known and strictly related to the estimates of the Bergman kernel along the diagonal (see \cite{tian,zel,ber1,ma2}). With a slight abuse of notation, we will denote by $\norm{\cdot}$ any norm induced by $h_d$.
\begin{lemma}(\cite[Proposition 1.5]{anc3})\label{asympeak} For any $x\in X$, let $s_0$ and $s_1$ be the peak sections defined in Definition \ref{defin peak}. Then, as $d\rightarrow +\infty$, we have the estimates $\norm{s_0(x)}=\frac{\sqrt{d}}{\sqrt{\pi}}(1+O(d^{-1}))$ and 
$\norm{\nabla s_1(x)}=\frac{d}{\sqrt{\pi}}(1+O(d^{-1}))$, where the error terms are uniform in $x\in X$.
\end{lemma} 

\subsection{Wronskian and higher moments}
Let $F$ and $E$ be respectively  degree $1$ and $0$ real holomorphic line bundles over $X$.
The purpose of this section is to prove Proposition \ref{moments}, which gives  key estimates of the higher moments of the random variable $(\alpha,\beta)\in \R H^0(X,F^d\otimes E)^2\mapsto \log \norm{\frac{\pi}{d^{3/2}}W_{\alpha\beta}(x)}$, where $W_{\alpha\beta}$ is the Wronskian, given by the following  
\begin{defn}\label{wr} Let $\nabla$ be a connection on $F^d\otimes E$. For any pair of real holomorphic global sections $(\alpha,\beta)\in \R H^0(X,F^d\otimes E)^2$, we denote by $W_{\alpha\beta}$ the Wronskian $\alpha\otimes\nabla\beta-\beta\otimes\nabla\alpha$, which is a real holomorphic global section of $F^{2d}\otimes E^2\otimes T^*_{X}$.
\end{defn}
\begin{oss}   The Wronskian $W_{\alpha\beta}$ does not depend on the choice of a connection on $F^d\otimes E$. Indeed, two connections $\nabla$ and $\nabla'$ on $F^d\otimes E$ differ by a $1$-form $\theta$, and then we have $$(\alpha\otimes\nabla\beta-\beta\otimes\nabla\alpha)-(\alpha\nabla'\beta-\beta\nabla'\alpha)=\alpha\otimes(\nabla-\nabla')\beta-\beta\otimes(\nabla-\nabla')\alpha=\alpha\otimes\beta\otimes\theta-\beta\otimes\alpha\otimes\theta=0.$$
\end{oss}
\begin{prop}\cite[Proposition 2.3]{anc3}\label{wronskiancritical} Let $F$ and $E$ be respectively  degree $1$ and $0$ real line bundles over $X$ and $(\alpha,\beta)\in \R H^0(X,F^d\otimes E)^2$ be a pair of sections without common zeros. A point $x\in X$ is a critical point of the map $u_{\alpha\beta}:x\in X\mapsto [\alpha(x):\beta(x)]\in\C\mathbb{P}^1$ if and only if it is a zero of the Wronskian $W_{\alpha\beta}$ defined in Definition \ref{wr}.
\end{prop}
\begin{prop}\label{moments} Let $X$ be a real algebraic curve equipped with a compatible volume form $\omega$ of total volume $1$ and let $F\in\Pic^1_{\R}(X)$. For any $B>0$ there exists an integer $d_B$ and a constant $c_B$ such that for any $E\in \Pic^0_{\R}(X)$, any $m\in\mathbb{N}$,  any $d\geq d_B$  and any point $x\in X$ with $\textrm{dist}(x,\R X)\geq B\frac{\log d}{\sqrt{d}}$, we have
$$\int_{(\alpha,\beta)\in \R H^0(X,F^d\otimes E)^2}\abs{\log \norm{\frac{\pi}{d^{3/2}}W_{\alpha\beta}(x)}}^m\textrm{d}\gamma_d(\alpha,\beta)\leq  c_B(m+1)!.$$
Here, $\textrm{dist}(\cdot,\cdot)$ is the distance in $X$ induced by $\omega$,
$\gamma_d$ is the Gaussian measure on $\R H^0(X,F^d\otimes E)^2$  constructed in Section \ref{secgauss} and $\norm{\cdot}$ denote the norm induced by the Hermitian metrics on $F$ and  $E$ given by Proposition \ref{metr}.
\end{prop}
\begin{proof} Let us consider the integral we want to estimate:
\begin{equation}\label{integrale da calcolare}
\int_{(\alpha,\beta)\in \R H^0(X,F^d\otimes E)^2}\abs{\log \norm{\frac{\pi}{d^{3/2}}W_{\alpha\beta}(x)}}^m\textrm{d}\gamma_d(\alpha,\beta).
\end{equation}
  First, remark that the function in the integral \eqref{integrale da calcolare} only depends on the $1$-jet of the sections $\alpha$ and $\beta$. We will then write the orthogonal decomposition $\R H^0(X,F^d\otimes E)=\R H_{2x}\oplus\R H_{2x}^{\perp}$, where  $\R H_{2x}$ is the space of real sections $s$ such that $s(x)=0$ and $\nabla s(x)=0$. As the Gaussian measure is a product measure, after the integration over the orthogonal of $\R H_{2x}^{\perp}\times \R H_{2x}^{\perp}$, we get that the integral \eqref{integrale da calcolare} is equal to
\begin{equation}\label{int sur l orthogo}
\int_{(\alpha,\beta)\in \R H_{2x}^{\perp}\times \R H_{2x}^{\perp}}\abs{\log \norm{\frac{\pi}{d^{3/2}}W_{\alpha\beta}(x)}}^m\textrm{d}\gamma_d\mid_{\R H_{2x}^{\perp}\times \R H_{2x}^{\perp}}(\alpha,\beta).
\end{equation}
 Using the notations of Section \ref{secpeak}, and in particular Definitions \ref{defn eval} and \ref{defn real eval}, let  $J_d: H_{2x}^{\perp}\rightarrow \R H_{2x}^{\perp}$ be the map $ (ev^{\R}_x\times ev^{\R}_{2x})^{-1}\circ (ev_x\times ev_{2x})$  and denote by $$I_d=J_d\times J_d:H_{2x}^{\perp}\times H_{2x}^{\perp}\rightarrow \R H_{2x}^{\perp}\times \R H_{2x}^{\perp}.$$ By changing of variables given by the isomorphism $I_d$, we get 
\begin{equation}\label{changing of variable}
\eqref{int sur l orthogo}=\int_{(\alpha,\beta)\in  H_{2x}^{\perp}\times H_{2x}^{\perp}}\abs{\log \norm{\frac{\pi}{d^{3/2}}W_{\alpha\beta}(x)}}^m(I_d^{-1})_*(\textrm{d}\gamma_d\mid_{\R H_{2x}^{\perp}\times \R H_{2x}^{\perp}})(\alpha,\beta).
\end{equation} 
By Corollary \ref{bounded determinant}, the maps $I_d$ and $I_d^{-1}$ have determinants bounded from above by a constant which only depends on $B$. In particular, there exists a constant $c_1$, depending only on $B$, such that 
\begin{equation}\label{int sur l ortho complex}
\eqref{changing of variable}\leq c_1 \int_{(\alpha,\beta)\in  H_{2x}^{\perp}\times  H_{2x}^{\perp}}\abs{\log \norm{\frac{\pi}{d^{3/2}}W_{\alpha\beta}(x)}}^m\textrm{d}\gamma^{\C}_d\mid_{ H_{2x}^{\perp}\times  H_{2x}^{\perp}}(\alpha,\beta)
\end{equation}
where $\gamma^{\C}_d$ is the complex Gaussian measure defined in Definition \ref{complexgaussian}.
In order to prove the result, we have to bound from above the quantity 
\begin{equation}\label{quanti da borner}
\int_{(\alpha,\beta)\in  H_{2x}^{\perp}\times  H_{2x}^{\perp}}\abs{\log \norm{\frac{\pi}{d^{3/2}}W_{\alpha\beta}(x)}}^m\textrm{d}\gamma^{\C}_d\mid_{ H_{2x}^{\perp}\times  H_{2x}^{\perp}}(\alpha,\beta)
\end{equation}
Let  $s_0$ and $s_1$ be the  peak sections at $x$ introduced in Definition \ref{defin peak} and we write $\alpha=a_0\sigma_0+a_1\sigma_1$ and $\beta=b_0\sigma_0+b_1\sigma_1$. We then have  $$\norm{W_{\alpha\beta}(x)}=\abs{a_0b_1-a_1b_0}\norm{(s_0\otimes\nabla s_1-s_1\otimes\nabla s_0)(x)}=\abs{a_0b_1-a_1b_0}\frac{d^{3/2}}{\pi}(1+O(d^{-c_2(B)})),$$ where the last equality follows from Proposition \ref{asympeak}.
This implies that the integral in \eqref{quanti da borner} equals
$$
\int_{\substack{a=(a_0,a_1)\in \C^2 \\ b=(b_0,b_1)\in\C^2}}\abs{\log \bigg(\abs{a_0b_1-a_1b_0}\norm{\frac{\pi}{d^{3/2}}(s_0\otimes\nabla s_1-s_1\otimes\nabla s_0)(x)}\bigg)}^m\frac{e^{-\abs{a}^2-\abs{b}^2}}{\pi^4}\textrm{d}a\textrm{d}b
$$
$$=\int_{\substack{a=(a_0,a_1)\in \C^2 \\ b=(b_0,b_1)\in\C^2}}\abs{\log \bigg(\abs{a_0b_1-a_1b_0}\bigg)}^m\frac{e^{-\abs{a}^2-\abs{b}^2}}{\pi^4}\big(1+O(d^{-c_3(B)})\big)\textrm{d}a\textrm{d}b
$$
\begin{equation}\label{firstinequality}
\leq 2\int_{\substack{a\in \C^2 \\ b\in\C^2}}\abs{\log\abs{a_0b_1-b_0a_1}}^m\frac{e^{-\abs{a}^2-\abs{b}^2}}{\pi^4}\textrm{d}a\textrm{d}b
\end{equation}
where the last inequality holds for $d\geq d_B$, for some $d_B$ large enough.\\

In the remaining part of the proof, we will estimate the last integral appearing in \eqref{firstinequality}. In order to do this, for any $a=(a_0,a_1)$ we make an unitary trasformation of $\C^2$ (of coordinates $b_0,b_1$) by sending the vector $(1,0)$ to $v_a=\frac{1}{\sqrt{|a_0|^2+|a_1|^2}}(a_0,a_1)$ and the vector $(0,1)$ to $w_a=\frac{1}{\sqrt{|a_0|^2+|a_1|^2}}(-\bar{a}_1,\bar{a}_0).$ We will write any vector of $\C^2$ as a sum $tv_a+sw_a $ with $s,t\in\C$. Under this change of variables, the integral appearing in (\ref{firstinequality}) becomes
\begin{equation}\label{secondinequality}
\leq 2\int_{\substack{a\in \C^2 \\ (s,t)\in\C^2}}\abs{\log\abs{s}\norm{a}}^m\frac{e^{-\abs{a}^2-\abs{s}^2-\abs{t}^2}}{\pi^4}\textrm{d}a\textrm{d}s\textrm{d}t=2\int_{\substack{a\in \C^2 \\ s\in\C}}\abs{\log\abs{s}\norm{a}}^m\frac{e^{-\abs{a}^2-\abs{s}^2}}{\pi^{3}}\textrm{d}a\textrm{d}s.
\end{equation}
We pass to polar coordinates $a=re^{i\theta}$, for $\theta\in S^3$ and $r\in\R_+$, and $s=\rho e^{i\phi}$, for $\phi\in S^1$ and $\rho\in\R_+$, and we obtain 
\begin{equation}\label{thirdinequality}
2\int_{\substack{a\in \C^2 \\ s\in\C}}\abs{\log\abs{s}\norm{a}}^m\frac{e^{-\abs{a}^2-\abs{s}^2}}{\pi^{3}}\textrm{d}a\textrm{d}s=8\int_{\substack{r\in \R_+ \\ \rho\in\R_+}}\abs{\log \rho r}^me^{-r^2-\rho^2}r^3\rho\textrm{d}r\textrm{d}\rho.
\end{equation}
Writing $\log \rho r= \log \rho +\log r$,  developing the binomial and using the triangular inequality, we obtain 
\begin{equation}\label{forthinequality}
(\ref{thirdinequality})\leq  8\int_{\substack{r\in \R_+ \\ \rho\in\R_+}}\sum_{k=0}^m\binom{m}{k}\abs{\log \rho}^k\abs{\log r}^{m-k}e^{-r^2-\rho^2}r^3\rho\textrm{d}r\textrm{d}\rho.
\end{equation}
Let us study the integrals   $\int_{\rho\in\R_+}\abs{\log \rho}^ne^{-\rho^2}\rho\textrm{d}\rho$ and $\int_{r\in\R_+}\abs{\log r}^ne^{-r^2}r^3\textrm{d}r$ .
To compute these two integrals, we will use the following formula obtained by integration by part:
\begin{equation}\label{logrecursion}
\int(\log x)^n\textrm{d}x=x\log x-n\int(\log x)^{n-1}\textrm{d}x, \hspace{2mm} n>0.
\end{equation}
\begin{itemize}
\item Computation of the integral 
$\int_{\rho\in\R_+}\abs{\log \rho}^ne^{-\rho^2}\rho\textrm{d}\rho.$
We write \begin{equation}\label{secondintegral}
\int_{\rho\in\R_+}\abs{\log \rho}^ne^{-\rho^2}\rho\textrm{d}\rho=\int_{\rho=0}^{1}(-\log \rho)^ne^{-\rho^2}\rho\textrm{d}\rho+\int_{\rho=1}^{\infty}(\log \rho)^ne^{-\rho^2}\rho\textrm{d}\rho.
\end{equation}
For the first term of this sum we have 
\begin{equation}\label{firstpartsecond}
\int_{\rho=0}^{1}(-\log \rho)^ne^{-\rho^2}\rho\textrm{d}\rho\leq \frac{\sqrt{2}}{2}\int_{\rho=0}^{1}(-\log \rho)^n\textrm{d}\rho =\frac{\sqrt{2}}{2}n!
\end{equation}
where we used first that $e^{-\rho^2}\rho\leq \frac{\sqrt{2}}{2}$ for $\rho\in [0,1]$ and then we used $n$ times the formula (\ref{logrecursion}).\\
For the second term of the sum in \eqref{secondintegral}, we use first the fact that $e^{-\rho^2}\rho\leq \frac{e^{-\frac{1}{\rho^2}}}{\rho^3}$ for any $\rho\geq 1$ and then the change $t=1/\rho$, to have 
\begin{equation}\label{new2}
\int_{\rho=1}^{\infty}(\log \rho)^ne^{-\rho^2}\rho\textrm{d}\rho\leq \int_{\rho=1}^{\infty}(\log \rho)^n\frac{e^{-\frac{1}{\rho^2}}}{\rho^3}\textrm{d}\rho\underset{t=1/\rho}{=}-\int_1^0(\log(1/t))^nte^{-t}\textrm{d}t=\int^1_0(-\log(t))^nte^{-t}\textrm{d}t.
\end{equation}
The last integral is the same as in (\ref{firstpartsecond}), so from (\ref{firstpartsecond}) and \eqref{new2} we obtain 
\begin{equation}\label{secondpartsecond}
\int_{\rho=1}^{\infty}(\log \rho)^ne^{-\rho^2}\rho\textrm{d}\rho\leq \frac{\sqrt{2}}{2}n!
\end{equation}
Putting  (\ref{firstpartsecond}) and (\ref{secondpartsecond}) in (\ref{secondintegral}), we obtain 
\begin{equation}\label{secondcomputation}
\int_{\rho\in\R_+}\abs{\log \rho}^ne^{-\rho^2}\rho\textrm{d}\rho\leq \sqrt{2}n!.
\end{equation}
\item Computation of the integral $\int_{r\in\R_+}\abs{\log r}^ne^{-r^2}r^3\textrm{d}r$. As before, we write \begin{equation}\label{firstintegral}
\int_{r\in\R_+}\abs{\log r}^ne^{-r^2}r^3\textrm{d}r=\int_{r=0}^1(-\log r)^ne^{-r^2}r^3\textrm{d}r+\int_{r=1}^{\infty}(\log r)^ne^{-r^2}r^3\textrm{d}r.
\end{equation}
For the first term of the sum, we get 
\begin{equation}\label{firstpartfirst}
\int_{r=0}^1(-\log r)^ne^{-r^2}r^3\textrm{d}r\leq (-1)^n\frac{\sqrt{2}}{\sqrt{3}}\int_{r=0}^1(\log r)^n\textrm{d}r=\frac{\sqrt{2}}{\sqrt{3}}n!
\end{equation}
where the first inequality follows from $e^{-r^2}r^3\leq \frac{\sqrt{2}}{\sqrt{3}}$, for $r\in [0,1]$, and the last equality is obtained using $n$ times the formula (\ref{logrecursion}).\\
For the second term of the sum in the right-hand side of (\ref{firstintegral}), we use integration by parts with respect to the functions $-\frac{1}{2}(\log r)^nr^2$ and $-2re^{-r^2}$ to obtain
\begin{equation}\label{new}\int_{s=1}^{\infty}(\log r)^ne^{-r^2}r^3\textrm{d}r=[-\frac{1}{2}(\log r)^nr^2e^{-r^2}]_{r=1}^{\infty}+\frac{n}{2}\int_{r=1}^{\infty}(\log r)^{n-1}re^{-r^2}\textrm{d}r+\int_{r=1}^{\infty}(\log r)^nre^{-r^2}\textrm{d}r.
\end{equation}
 As $[-\frac{1}{2}(\log r)^nr^2e^{-r^2}]_{r=1}^{\infty}=0$ we obtain,  by using \eqref{secondpartsecond} in \eqref{new}, that \begin{equation}\label{secondpartfirst}
\int_{s=1}^{\infty}(\log r)^ne^{-r^2}r^3\textrm{d}r\leq \frac{3\sqrt{2}}{4}n!.
\end{equation}
Putting (\ref{firstpartfirst}) and (\ref{secondpartfirst}) in (\ref{firstintegral}), we get 
\begin{equation}\label{firstcomputation}
\int_{r\in\R_+}\abs{\log r}^ne^{-r^2}r^3\textrm{d}s\leq \frac{4\sqrt{6}+9\sqrt{2}}{12}n!.
\end{equation}
\end{itemize}
Now, we use (\ref{secondcomputation}) and (\ref{firstcomputation})   and  we obtain the following estimate:
\begin{equation}\label{fifthinequality}\int_{\substack{r\in \R_+ \\ \rho\in\R_+}}\sum_{k=0}^m\binom{m}{k}\abs{\log \rho}^k\abs{\log r}^{m-k}e^{-r^2-s^2}r^3\rho\textrm{d}r\textrm{d}\rho\leq
\frac{4\sqrt{3}+9}{6}\sum_{k=0}^m\binom{m}{k}k!(m-k)!\leq  \frac{4\sqrt{3}+9}{6}(m+1)!.
\end{equation}
Putting (\ref{fifthinequality}) in (\ref{forthinequality}) and using  (\ref{thirdinequality}) , (\ref{secondinequality}) and (\ref{firstinequality}), we obtaine the desired estimate for \eqref{quanti da borner}, hence the result.
\end{proof}
\section{Proof of Theorem \ref{rarefaction}}\label{secproof}
In this section, we prove our main result. We follow the notations of Sections \ref{ramframwork} and \ref{secmoment}.
\begin{prop}\label{largedev} Let $X$ be a real algebraic curve equipped with a compatible volume form $\omega$ of total volume $1$ and let $F\in\Pic^1_{\R}(X)$. Fix   a sequence of positive real numbers $(a_d)_d$. Then, for any $B>0$ there exists  $d_B\in\mathbb{N}$  and a constant $c_B$ such that, for any $E\in\Pic_{\R}^0(X)$, any $d\geq d_B$ and any  sequence of smooth functions $(\varphi_d)_d$ with $\textrm{dist}(\textrm{supp}(\varphi_d),\R X)\geq B\frac{\log d}{\sqrt{d}}$, the following holds
$$\gamma_{F^d\otimes E}\bigg\{(\alpha,\beta)\in\R H^0(X,F^d\otimes E)^2, \abs{\int_X\log \bigg(\frac{\pi}{d^{3/2}}\norm{W_{\alpha\beta}(x)}\bigg)\partial\bar{\partial}\varphi_d}\geq a_d\bigg\}$$ $$\leq c_B\exp(-\frac{a_d }{2\norm{\partial\bar{\partial}\varphi_d}_{\infty}\Vol(\textrm{Supp}(\partial\bar{\partial}\varphi_d))}).$$
Here, $\textrm{dist}(\cdot,\cdot)$ is the distance in $X$ induced by $\omega$,
$\gamma_{F^d\otimes E}$ is the Gaussian measure on $\R H^0(X,F^d\otimes E)^2$  constructed in Section \ref{secgauss} and $\norm{\cdot}$ denote the pointwise norm induced by the Hermitian metrics on $F$ and  $E$ given by Proposition \ref{metr}.
\end{prop}
\begin{proof}
For any $t_d>0$, let us denote 
\begin{equation}\label{defexp}
\exp(t_d\abs{\int_X\log\bigg(\frac{\pi}{d^{3/2}}\norm{W_{\alpha\beta}(x)}\bigg)\partial\bar{\partial}\varphi_d})=\sum_{m=0}^{\infty}\frac{t_d^m}{m!}\abs{\int_X\log\bigg(\frac{\pi}{d^{3/2}}\norm{W_{\alpha\beta}(x)}\bigg)\partial\bar{\partial}\varphi_d}^m.
\end{equation}
Remark that 
\begin{equation}\label{equivalence of quantities}
\abs{\int_X\log\bigg(\frac{\pi}{d^{3/2}}\norm{W_{\alpha\beta}(x)}\bigg)\partial\bar{\partial}\varphi_d}\geq a_d d\Leftrightarrow\exp(t_d\abs{\int_X\log\bigg(\frac{\pi}{d^{3/2}}\norm{W_{\alpha\beta}(x)}\bigg)\partial\bar{\partial}\varphi_d})\geq e^{t_da_d}
\end{equation}
so that, by Markov inequality, we have
$$
\gamma_{F^d\otimes E}\bigg\{(\alpha,\beta)\in\R H^0(X,F^d\otimes E)^2, \abs{\int_X\log\frac{\pi}{d^{3/2}}\norm{W_{\alpha\beta}(x)}\partial\bar{\partial}\varphi_d}\geq a_d\bigg\}\leq 
$$
\begin{equation}\label{markov}
e^{-t_d a_d}\int_{\R H^0(X,F^d\otimes E)^2}\exp(t_d\abs{\int_X\log\bigg(\frac{\pi}{d^{3/2}}\norm{W_{\alpha\beta}(x)}\bigg)\partial\bar{\partial}\varphi_d})\textrm{d}\gamma_{F^d\otimes E}.
\end{equation}
Now, we have 
\begin{equation}\label{holder1}
\abs{\int_X\log\bigg(\frac{\pi}{d^{3/2}}\norm{W_{\alpha\beta}(x)}\bigg)\partial\bar{\partial}\varphi_d}^m\leq
\norm{\partial\bar{\partial}\varphi_d}^m_{\infty}\abs{\int_{\textrm{Supp}(\partial\bar{\partial}\varphi_d)}\log\bigg(\frac{\pi}{d^{3/2}}\norm{W_{\alpha\beta}(x)}\bigg)\omega}^m.
\end{equation}
We then apply H\"older inequality with $m$ and $m/(m-1)$ for the functions $\log \bigg(\frac{\pi}{d^{3/2}}\norm{W_{\alpha\beta}(x)}\bigg)$ and $1$, so that 
\begin{equation}\label{holder2}
(\ref{holder1})\leq \norm{\partial\bar{\partial}\varphi_d}^m_{\infty}\Vol(\textrm{Supp}(\partial\bar{\partial}\varphi_d))^{m-1}\int_{\textrm{Supp}(\partial\bar{\partial}\varphi_d)}\abs{\log\bigg(\frac{\pi}{d^{3/2}}\norm{W_{\alpha\beta}(x)}\bigg)}^m\omega.
\end{equation}
By Proposition \ref{moments}, there exists $d_B\in\mathbb{N}$ and a positive constant $c_B$ such that for any $d\geq d_B$ we get
\begin{equation}\label{holder3}
\textrm{right-hand side\hspace{1.5mm}of\hspace{1.5mm}}(\ref{holder2})\leq \norm{\partial\bar{\partial}\varphi_d}^m_{\infty}\Vol(\textrm{Supp}(\partial\bar{\partial}\varphi_d))^{m}c_B(m+1)!.
\end{equation}
Then, by (\ref{markov}), (\ref{defexp}) and (\ref{holder3}), we have
$$\gamma_{F^d\otimes E}\bigg\{(\alpha,\beta)\in\R H^0(X,F^d\otimes E)^2, \abs{\int_X\log\frac{\norm{W_{\alpha\beta}(x)}}{d^{3/2}}\partial\bar{\partial}\varphi_d}\geq a_d\bigg\}\leq$$
\begin{equation}\label{holdfinal}
e^{-t_d a_d}c_B\sum_{m=0}^{\infty}(m+1)\bigg(\norm{\partial\bar{\partial}\varphi_d}_{\infty}\cdot\Vol(\textrm{Supp}(\partial\bar{\partial}\varphi_d))\bigg)^mt_d^m.
\end{equation}
Now, we have the identity $\sum_{m=0}^{\infty}(m+1)x^m=\frac{\textrm{d}}{\textrm{d}x}\sum_{m=1}^{\infty}x^m=\frac{\textrm{d}}{\textrm{d}x}\big(\frac{1}{(1-x)}-1\big)=\frac{1}{(1-x)^2}$, so that the right hand side in \eqref{holdfinal} equals 
\begin{equation}\label{final}
\frac{c_B\exp(-t_d a_d)}{\big(1-t_d\norm{\partial\bar{\partial}\varphi_d}_{\infty}\cdot\Vol(\textrm{Supp}(\partial\bar{\partial}\varphi_d)) \big)^2}
\end{equation}
Putting $t_d=\big(2\norm{\partial\bar{\partial}\varphi_d}_{\infty}\cdot\Vol(\textrm{Supp}(\partial\bar{\partial}\varphi_d))\big)^{-1}$, we get the result.
\end{proof}
\begin{lemma}[Lemma 2 of \cite{gwexp}]\label{cutoff} There exist positive constants $C_i$, $i\in\{1,\dots, 4\}$, and  a family of cutoff functions $\chi_t:X\rightarrow [0,1]$, defined for $t\in (0,t_0]$, for some $t_0>0$, such that 
\begin{enumerate}
\item $\Vol(\textrm{supp}(\partial\bar{\partial}\chi_t))\leq C_1t$;
\item $\Vol(X\setminus\chi_t^{-1}(1))\leq C_2t$;
\item $\norm{\partial\bar{\partial}\chi_t}_{L^{\infty}}\leq C_3t^{-2}$;
\item $\textrm{dist}(\textrm{supp}(\chi_t),\R X))\geq C_4t$.
\end{enumerate}
\end{lemma}
We now prove the following fiberwise version of Theorem \ref{rarefaction}.
\begin{thm}\label{rarefiber} Let $\ell(d)$ be a sequence of positive real numbers such that $\ell(d)\geq B(\log d)$ for some $B>0$. Then there exist positive constants $c_1$ and $c_2$ such that 
$$\mu_{F^d\otimes E}\big\{u\in \mathcal{M}_d^{\R}(X,F^d\otimes E), \#(\Crit(u)\cap \R X)\geq \ell(d)\sqrt{d}\big\}\leq c_1e^{-c_2\ell(d)^2}.$$
Here, $\mu_{F^d\otimes E}$ is the probability measure defined in Definition \ref{measurefiber} and $ \mathcal{M}_d^{\R}(X,F^d\otimes E)$ is defined in Definition \ref{rhfs}.
\end{thm}
\begin{proof} For any pair of real global sections $(\alpha,\beta)\in \R H^0(X,F^d\otimes E)^2$ without common zeros, let $u_{\alpha\beta}$ be the real branched covering defined by $x\mapsto[\alpha(x):\beta(x)]$.
Consider the set 
\begin{equation}\label{cones}
\mathcal{C}_{\ell(d)}\doteqdot\{(\alpha,\beta)\in \R H^0(X,F^d\otimes E)^2, \#(\Crit(u_{\alpha\beta})\cap \R X)\geq \ell(d)\sqrt{d}\big\}.
\end{equation}
Remark that this set is a  cone in $\R H^0(X,F^d\otimes E)^2$.  By Proposition \ref{vs}, this implies that the Gaussian measure of $\mathcal{C}_{\ell(d)}$ equals the Fubini-Study measure of its projectivization, which is exactly the measure we want to estimate. In order to obtain the result, we will then compute the Gaussian measure of  the cone (\ref{cones}). Moreover,  by Proposition \ref{wronskiancritical}, we have that $x\in \Crit(u_{\alpha\beta}) $ if and only if $W_{\alpha\beta}(x)=0$, so that, in order to compute $\#\Crit(u_{\alpha\beta})$, we can compute the number of zeros of $W_{\alpha\beta}$. To do this, we will use Poincar\'e-Lelong formula, that is the following equality between currents
\begin{equation}\label{poincarelelong}
\omega_d- \sum_{x\in\{W_{\alpha\beta}=0\}}\delta_x =\frac{1}{2\pi i}\partial\bar{\partial}\log \norm{W_{\alpha\beta}},
\end{equation}
 where $\norm{\cdot}$ is the (induced) metric on $F^{2d}\otimes E^2\otimes T^*_{X}$ given by Proposition \ref{metr} and  $\omega_d$ is the corresponding curvature form.
Remark that $\omega_d$ equals $2d\cdot\omega+O(1)$ (the term $2d\cdot\omega$ comes from the curvature form of $F^{2d}\otimes E^{2}$ and the term $O(1)$ from the curvature form of $T^*_{X}$). Moreover, remark that  the Hermitian metric $\frac{\pi}{d^{3/2}}\norm{\cdot}$ has the same curvature of the Hermitian metric $\norm{\cdot}$, because the curvature form is not affected by a multiplicative constant.
Then,  Poincar\'e-Lelong formula \eqref{poincarelelong}, can also be read \begin{equation}\label{poincareinthiscase}
 2d\cdot\omega+O(1)- \sum_{x\in\{W_{\alpha\beta}=0\}}\delta_x =\frac{1}{2\pi i}\partial\bar{\partial}\log \bigg(\frac{\pi}{d^{3/2}}\norm{W_{\alpha\beta}}\bigg)
 \end{equation}
where the equality is in the sense of currents. We will apply \eqref{poincareinthiscase} for the functions $\chi_{t_d}$ given by Lemma \ref{cutoff}, for $t_d=\frac{\ell(d)}{4C_2\sqrt{d}}$, where $C_2$ is the constant appearing in Lemma \ref{cutoff}. By \eqref{poincareinthiscase}, we then get
\begin{equation}\label{poincaresub0}
\frac{1}{2\pi}\abs{\int_X\log \bigg(\frac{\pi}{d^{3/2}}\norm{W_{\alpha\beta}}\bigg)\partial\bar{\partial}\chi_{t_d}}\geq \abs{2d\big(1-\frac{\ell(d)}{4\sqrt{d}}\big)+O(1)-\sum_{x\in\{W_{\alpha\beta}=0\}}\chi_{\frac{\ell(d)}{\sqrt{d}}}(x)}.
\end{equation}
Remark that, for any pair of real global sections $(\alpha,\beta)$ in the cone $\mathcal{C}_{\ell(d)}$ defined in \eqref{cones},  we have 
\begin{equation}\label{poincaresub1}
\sum_{x\in\{W_{\alpha\beta}=0\}}\chi_{\frac{\ell(d)}{\sqrt{d}}}(x)\leq 2d+2g-2-\ell(d)\sqrt{d},
\end{equation}
 where $g$ is the genus of $X$. Then, putting \eqref{poincaresub1} in \eqref{poincaresub0}, we get
$$\frac{1}{2\pi}\abs{\int_X\log \bigg(\frac{\pi}{d^{3/2}}\norm{W_{\alpha\beta}}\bigg)\partial\bar{\partial}\chi_{\frac{\ell(d)}{\sqrt{d}}}}\geq \frac{1}{2}\ell(d)\sqrt{d}+O(1),$$
for any $(\alpha,\beta)\in \mathcal{C}_{\ell(d)}$.
Then, for $d$ large enough,  the cone (\ref{cones}) is included in the set
$$\bigg\{(\alpha,\beta)\in \R H^0(X,F^d\otimes E)^2, \abs{\int_X\log\bigg(\frac{\pi}{d^{3/2}}\norm{W_{\alpha\beta}}\bigg)\partial\bar{\partial}\chi_{\frac{\ell(d)}{\sqrt{d}}}}\geq \ell(d)\sqrt{d}\bigg\}.$$
The result then follows from Proposition \ref{largedev} and Lemma \ref{cutoff}.
\end{proof}
\begin{proof}[Proof of Theorem \ref{rarefaction}]
We fix a degree $1$ real holomorphic line bundle $F$ over $X$, so that for any $L\in\Pic_{\R}^d(X)$ there exists an unique degree $0$  real holomorphic line bundle $ E\in \Pic^0_{\R}(X)$ such that $L=F^d\otimes E$. The result then follows by integrating the inequality appearing in Theorem  \ref{rarefiber} along the compact base $\Pic^0_{\R}(X)\simeq \Pic^d_{\R}(X)$ (the last isomorphism is given by the choice of the degree $1$ real line bundle $F$).
\end{proof}
 

\bibliographystyle{plain}
\bibliography{biblio}
\end{document}